\documentclass[11pt,reqno,a4paper]{amsart}

\usepackage{a4wide}
\usepackage{dsfont} 	
\usepackage{amssymb}
\usepackage{graphicx}	
\usepackage{mathabx} 	
\usepackage{upgreek}
\usepackage{enumerate}
\usepackage[all]{xy}

\newtheorem{proposition}{Proposition}[section]
  \newtheorem{theorem}[proposition]{Theorem}
  \newtheorem{lemma}[proposition]{Lemma}
  \newtheorem{corollary}[proposition]{Corollary}
\theoremstyle{definition}
  
  \newtheorem{remark}[proposition]{Remark}

\newcommand{\cst}{\ifmmode\mathrm{C}^*\else{$\mathrm{C}^*$}\fi}
\newcommand{\st}{\;\vline\;}
\newcommand{\tens}{\otimes}
\newcommand{\vtens}{\,\bar{\otimes}\,}
\newcommand{\I}{\mathds{1}}
\newcommand{\comp}{\!\circ\!}
\newcommand{\id}{\mathrm{id}}
\newcommand{\is}[2]{\left\langle#1\,\vline\,#2\right\rangle}
\newcommand{\qqquad}{\quad\qquad}
\newcommand{\hh}[1]{\widehat{#1}}

\newcommand{\NN}{\mathbb{N}}
\newcommand{\RR}{\mathbb{R}}
\newcommand{\GG}{\mathbb{G}}
\newcommand{\HH}{\mathbb{H}}
\newcommand{\KK}{\mathbb{K}}

\newcommand{\sN}{\mathsf{N}}
\newcommand{\sJ}{\mathsf{J}}

\newcommand{\gN}{\mathfrak{N}}

\newcommand{\uu}{\text{\rm\tiny{u}}}

\DeclareMathOperator{\B}{\mathsf{B}}
\DeclareMathOperator{\C}{C}
\DeclareMathOperator{\M}{\mathsf{M}}
\DeclareMathOperator{\Dom}{\mathsf{D}}

\newcommand{\staru}{\coasterisk}
\newcommand{\cleq}{\preccurlyeq}
\newcommand{\ww}{\mathrm{W}}

\newcommand{\Ww}{\mathds{W}}
\newcommand{\wW}{\text{\reflectbox{$\Ww$}}\:\!} 
\newcommand{\dd}[1]{\widetilde{#1}}
\DeclareMathOperator{\Idem}{Idem}
\DeclareMathOperator{\Idemn}{Idem_{\text{\tiny{\rm{nor}}}}}
\DeclareMathOperator{\Linf}{\mathnormal{L}^\infty\;\!\!}
\DeclareMathOperator{\Ltwo}{\mathnormal{L}^2\;\!\!}

\numberwithin{equation}{section}

\DeclareMathAlphabet{\mathpzc}{OT1}{pzc}{m}{it}

\begin{document}

\author{Pawe{\l} Kasprzak}
\address{Department of Mathematical Methods in Physics, Faculty of Physics, University of Warsaw, Poland}
\email{pawel.kasprzak@fuw.edu.pl}

\author{Piotr M.~So{\l}tan}
\address{Department of Mathematical Methods in Physics, Faculty of Physics, University of Warsaw, Poland}
\email{piotr.soltan@fuw.edu.pl}

\title[Lattice of Idempotent States]{Lattice of Idempotent States on a Locally Compact Quantum Group}

\keywords{idempotent state, locally compact quantum group, open quantum subgroup}

\subjclass[2010]{Primary: 46L89, Secondary: 43A05, 20G42}


\begin{abstract}
We study lattice operations on the set of idempotent states on a locally compact quantum group corresponding to the operations of intersection of compact subgroups and forming the subgroup generated by two compact subgroups. Normal ($\sigma$-weakly continuous) idempotent states are investigated and a duality between normal idempotent states on  a locally compact quantum group $\mathbb{G}$ and on its dual $\widehat{\mathbb{G}}$ is established. Additionally we analyze the question when a left coideal corresponding canonically to an idempotent state is finite dimensional and give a characterization of normal idempotent states on compact quantum groups.
\end{abstract}

\maketitle

\section{Introduction}\label{intro}

Idempotent states on locally compact quantum groups have been investigated in a number of papers, e.g.~\cite{Pal,FrSk09-1,FrSk09-2,BFS,SaSk12,FST,Salmi,SaSk,KK,FK}. The main idea behind these investigations was based on the classical result of Kawada and It\^o which establishes a bijection between idempotent states on a classical locally compact group $G$ and compact subgroups of $G$ with the state given by integration with respect to the Haar measure of the subgroup. As shown first by Pal in \cite{Pal} this bijection is no longer there when we replace $G$ by a quantum group (even a finite one). Nevertheless the set of idempotent states on a locally compact quantum group does carry a rich structure which has been explored by many authors. In most approaches the idea that each idempotent state corresponds to a certain ``subgroup-like'' object has been (more or less explicitly) one of the main motivations for investigating these objects. It is worth mentioning that a similar idea was explored in \cite{BBS} where the ``subgroup-like'' object described by a certain Hilbert space vector was called ``pr\'e sous-groupe.''

In this paper we will also treat idempotent states as corresponding to certain -- purely virtual -- objects which we choose to call ``quasi-subgroups'' and we will study properties of the set of idempotent states from the point of view of lattice operations usually performed on subgroups (see Section \ref{Lattice}). Let us mention that for finite quantum groups this program was carried out in \cite{FrSk09-1}. Our approach is that of general locally compact quantum groups in the sense of Kustermans and Vaes (\cite{KV}), but favoring the right Haar measure as in \cite{mnw}.

Let us remark that notions related to what we call a quasi-subgroup have appeared also in the context of uncertainty principles for quantum groups (\cite{uncert,young,uncertKac}). In particular some of our results have analogs phrased in terms of group-like projections as well as so called ``biprojections''. However, since the papers \cite{uncert,young,uncertKac} do not make the connection with idempotent states, the proofs tend to be much less direct and involve sophisticated tools which are not needed in our approach.

Throughout the paper $\GG$ will denote a locally compact quantum group and according to standard conventions the symbols $\C_0(\GG)$, $\C_0^\uu(\GG)$ and $\Linf(\GG)$ stand for the corresponding reduced and universal \cst-algebras and the von Neumann algebra of ``non-commutative functions'' on $\GG$.

The right Haar measure on $\GG$ will be denoted by $\psi$ and $\Ltwo(\GG)$ will denote the GNS Hilbert space of $\psi$. We will use $\eta$ to denote the corresponding GNS map. In particular $\Dom(\eta)$ coincides with $\gN_\psi$, i.e.~the set of square integrable elements of $\Linf(\GG)$. For a vector subspace $\sN\subset\Linf(\GG)$ we will write $\Ltwo(\sN)$ for the closure of $\eta\bigl(\sN\cap\Dom(\eta)\bigr)$ in $\Ltwo(\GG)$ (this subspace may be $\{0\}$).

Throughout the paper $\Delta:\Linf(\GG)\to\Linf(\GG)\vtens\Linf(\GG)$ will be the comultiplication and $\Delta^\uu$ will denote the comultiplication on the universal \cst-algebra $\C_0^\uu(\GG)$. We will write $S$ for the antipode and $R$ and $\tau=(\tau_t)_{t\in\RR}$ for the unitary antipode and scaling group of $\GG$ (see \cite{KV,mn,mnw} for definitions and e.g.~\cite{DKSS} for the details of notation). The objects related to the dual locally compact quantum group $\hh{\GG}$ will be denoted by appropriate symbols decorated with hats, e.g.~$\hh{\Delta}$ or $\hh{\psi}$. We will also use notation introduced originally in \cite{DKSS} according to which $\ww\in\M(\C_0(\hh{\GG})\tens\C_0(\GG))\subset\Linf(\hh{\GG})\vtens\Linf(\GG)$ is the Kac-Takesaki operator of $\GG$ (\cite{mnw}, it is the \emph{right regular representation} in the terminology of \cite{KV}) and $\wW\in\M(\C_0(\hh{\GG})\tens\C_0^\uu(\GG))$ is its ``half-lifted'' version (cf.~\cite{univ,DKSS}).

Let us recall the notion of a \emph{left coideal} which will be of crucial importance throughout the paper: we say that a von Neumann subalgebra $\sN\subset\Linf(\GG)$ is a left \emph{coideal} if $\Delta(\sN)\subset\Linf(\GG)\vtens\sN$. It follows that $\Delta$ defines an \emph{action} of $\GG$ on $\sN$. A left coideal $\sN$ is \emph{integrable} if $\psi$ is semifinite on $\sN$. In fact this condition is implied by existence of just one non-zero integrable positive element of $\sN$ (cf.~\cite[Proposition 3.3]{KaKhSo}).

In Section \ref{IdempSt} we will recall the notion of an idempotent state and introduce various objects related to one. We also prove some of their basic properties and summarize results of the papers \cite{KK,FK} which are used later throughout the paper. Next, in Section \ref{Lattice} we describe the lattice operations on the set of idempotent states which respectively correspond to intersection of quasi-subgroups and to the quasi-subgroup generated by two quasi-subgroups (we leave it to the reader to check that in case of idempotent states of Haar type (see  Section \ref{Lattice}) these operations indeed coincide with certain operations on compact quantum subgroups). In Subsection \ref{modLaw} we prove a form of the modular law for the lattice operations discussed earlier.

Section \ref{openSect} is devoted to the study of idempotent states which are in addition \emph{normal} (arise from $\sigma$-weakly continuous states on $\Linf(\GG)$). By definition they correspond to \emph{open} quasi-subgroups. We find an interesting duality for such objects: there is a bijective correspondence between normal idempotent states on $\GG$ and normal idempotent states on $\hh{\GG}$ and performing this passage twice yields the original state. Among other results of Section \ref{openSect} we provide characterization of open quasi-subgroups of \emph{compact} quantum groups by finite dimensionality of the corresponding coideal (mimicking the space of cosets) and give several other results about finite dimensional coideals. Finally we discuss the lattice operations of Section \ref{Lattice} for normal idempotent states.

\section{Idempotent states, coideals and group-like projections}\label{IdempSt}

Let $\omega$ be a state on $\C_0^\uu(\GG)$. Recall that $\C_0^\uu(\GG)^*$ carries a structure of a Banach algebra with multiplication given by convolution which we will denote by
\[
\C_0^\uu(\GG)^*\times\C_0^\uu(\GG)^*\ni(\omega,\mu)\longmapsto\omega\staru\mu=(\omega\tens\mu)\comp\Delta^\uu\in\C_0^\uu(\GG)^*.
\]
We say that a state $\omega$ on $\C_0^\uu(\GG)$ is \emph{idempotent} if $\omega\staru\omega=\omega$. The set of idempotent states on $\C_0^\uu(\GG)$ will be denoted by $\Idem(\GG)$.

Now take $\omega\in\Idem(\GG)$. The formula
\[
E_\omega(x)=\omega\staru{x}=(\id\tens\omega)\bigl(\wW(x\tens\I)\wW^*\bigr),\qqquad{x}\in\Linf(\GG),
\]
defines a normal conditional expectation on $\Linf(\GG)$ and we will write $\sN_\omega$ to denote its range:
\[
\sN_\omega=\bigl\{\omega\staru{x}\st{x}\in\Linf(\GG)\bigr\}.
\]
Moreover $\sN_\omega$ is a $\tau$-invariant integrable left coideal in $\Linf(\GG)$. Let
\[
P_\omega=(\id\tens\omega)\wW\in\Linf(\hh{\GG}).
\]
It turns out that $P_\omega$ is a $\hh{\tau}$-invariant group-like projection in $\Linf(\hh{\GG})$ (which means that $\hh{\Delta}(P_\omega)(\I\tens{P_\omega})=P_\omega\tens{P_\omega}$) and furthermore it is the orthogonal projection onto $\Ltwo(\sN_\omega)$. In fact
\begin{equation}\label{eq1}
\eta\bigl(E_\omega(x)\bigr)=P_\omega\eta(x),\qqquad{x}\in\Dom(\eta).
\end{equation}

The papers \cite{KK} and \cite{FK} establish bijective correspondences between the sets of
\begin{itemize}
\item idempotent states $\omega$ on $\C_0^\uu(\GG)$,
\item integrable $\tau$-invariant left coideals $\sN$ in $\Linf(\GG)$,
\item group-like projections $P\in\Linf(\hh{\GG})$ invariant under the scaling group,
\end{itemize}
under which $\omega$ corresponds to $\sN_\omega$ and $P_\omega$.

Let $\mu,\nu\in\Idem(\GG)$. We will say that $\nu$ \emph{dominates} $\mu$ if $\mu\staru\nu=\nu$ and in this case we write $\mu\cleq\nu$. Note that this is clearly equivalent to $E_\mu\comp{E_\nu}=E_\nu$. The relation $\cleq$ defines a partial order on $\Idem(\GG)$. Note that our convention is the opposite to the one used in \cite{FrSk09-1}, as our notation aims to indicate the containment of quasi-subgroups rather than the natural order on idempotents in the Banach algebra $\C_0^\uu(\GG)^*$. Our convention agrees with the analogous notation in \cite{BBS}.

The next simple lemma will be used throughout the paper:

\begin{lemma}\label{lemcomp}
Let $\mu,\nu\in\Idem(\GG)$. Then the following are equivalent:
\begin{enumerate}[\indent\rm(1)]
\item $\mu\cleq\nu$,
\item $E_\mu\comp{E_\nu}=E_\nu$,
\item $\sN_\nu\subset\sN_\mu$,
\item $P_\nu\leq{P_\mu}$.
\end{enumerate}
\end{lemma}

\begin{proof}
We already noted that $\mu\cleq\nu$ means simply $E_\mu\comp{E_\nu}=E_\nu$ which is equivalent to $\sN_\nu\subset\sN_\mu$. The latter condition implies that $\Ltwo(\sN_\nu)\subset\Ltwo(\sN_\mu)$ which, in turn, is equivalent to $P_\nu\leq{P_\mu}$. Conversely, suppose $P_\nu\leq{P_\mu}$, i.e.~${P_\mu}{P_\nu}=P_\nu$. Then for any $x\in\Dom(\eta)$
\[
\eta\bigl(E_\nu(x)\bigr)=P_\nu\eta(x)={P_\mu}{P_\nu}\eta(x)={P_\mu}\eta\bigl(E_\nu(x)\bigr)=\eta\bigl(E_\mu(E_\nu(x))\bigr),
\]
so $E_\nu=E_\mu\comp{E_\nu}$.
\end{proof}

We let $\Idem_0(\GG)$ denote the set $\Idem(\GG)\cup\{0\}$ and note that if the correspondences above were to be (albeit artificially) extended to $\Idem_0(\GG)$ then $0\in\Idem_0(\GG)$ would have to be associated with any non-integrable $\tau$-invariant left coideal in $\Linf(\GG)$. Moreover the partial order $\cleq$ extends to $\Idem_0(\GG)$ with $0$ being the largest element.

\section{The lattice of compact quasi-subgroups}\label{Lattice}

Let $\GG$ be a locally compact quantum group and let $\KK$ be a compact quantum subgroup of $\GG$. Then we have a surjective $*$-homomorphism $\C_0^\uu(\GG)\to\C^\uu(\KK)$ and composing this homomorphism with the Haar measure of $\KK$ yields an element of $\Idem(\GG)$ (cf.~\cite{DKSS}). Such idempotent states are referred to as states of \emph{Haar type}. As mentioned in Section \ref{intro}, the classical Kawada-It\^o theorem (\cite[Theorem 3]{KI}) says that in case $\GG$ is classical, these are the only examples of idempotent states. Furthermore, as demonstrated in \cite{Pal}, idempotent states on quantum groups are not necessarily of Haar type. Nevertheless it is tempting to treat each idempotent state on a locally compact quantum group $\GG$ as corresponding to a ``virtual'' compact quantum subgroup $\KK$ of $\GG$. We will therefore say that an idempotent state $\omega$ on $\C_0^\uu(\GG)$ corresponds to a \emph{compact quasi-subgroup} of $\GG$. In this convention the coideal $\sN_\omega$ corresponds to the von Neumann algebra of essentially bounded functions on the set of cosets of the ``quasi-subgroup.'' Let us emphasize, however, that the terminology of quasi-subgroups introduced above is of purely linguistic nature and does not involve any new mathematical objects.

It turns out that one can introduce operations on idempotent states which, in case of idempotent states of Haar type, correspond to intersections of subgroups and taking a subgroup generated by two subgroups supporting the states. As we already mentioned in Section \ref{intro} this program was carried out for finite quantum groups in \cite[Section 5]{FrSk09-1}. We will use notation similar to the one used in that paper, but our symbols ``$\wedge$'' and ``$\vee$'' will correspond to intersection of quasi-subgroups and quasi-subgroup generated by two quasi-subgroups respectively. This is the opposite convention to the one used in \cite{FrSk09-1}.

\subsection{Intersection of quasi-subgroups}

Let $\omega,\mu\in\Idem(\GG)$. Then the von Neumann subalgebra $\sN_\omega\vee\sN_\mu$ generated by $\sN_\omega$ and $\sN_\mu$ is clearly a $\tau$-invariant left coideal which, furthermore, is integrable (indeed, a coideal is integrable if it contains at least one non-zero integrable element, cf.~\cite[Proposition 3.3]{KK}). Therefore it must be of the form $\sN_\nu$ for some $\nu\in\Idem(\GG)$. This idempotent state will be denoted by $\omega\wedge\mu$ and the corresponding quasi-subgroup will be referred to as the \emph{intersection} of the quasi-subgroups corresponding to $\omega$ and $\mu$. Obviously the operation
\[
(\omega,\mu)\longmapsto\omega\wedge\mu
\]
is commutative and associative.

\subsection{Quasi-subgroup generated by two quasi-subgroups}

\begin{theorem}\label{vee}
Let $\omega,\mu\in\Idem(\GG)$. Then
\begin{enumerate}[\indent\rm(1)]
\item  $\Ltwo(\sN_\omega)\cap\Ltwo(\sN_\mu)=\Ltwo(\sN_\omega\cap\sN_\mu)$,
\item the sequence $\bigl((\omega\staru\mu)^{\staru{n}}\bigr)_{n\in\NN}$ is weak$^*$ convergent to $\nu\in\Idem_0(\GG)$,
\item\label{vee3} $\Ltwo(\sN_\omega)\cap\Ltwo(\sN_\mu)\neq\{0\}$ if and only if $\nu\neq{0}$  which is further equivalent to the fact that the coideal $\sN_\omega\cap\sN_\mu$ is non-zero and integrable; moreover, in this case, $\nu$ is the idempotent state corresponding to the $\tau$-invariant integrable left coideal $\sN_\omega\cap\sN_\mu$.
\end{enumerate}
\end{theorem}

\begin{proof}
Denote $\nu_n=(\omega\staru\mu)^{\staru{n}}$. We have
\begin{equation}\label{prod}
(\id\tens\nu_n)(\wW)=(P_{\omega}P_{\mu})^n.
\end{equation}
The right hand side of \eqref{prod} strongly converges to the projection $P$ onto $\Ltwo(\sN_\omega)\cap\Ltwo(\sN_\mu)$. In particular, for all vector functionals $\upomega_{\xi}$ (with $\xi\in\Ltwo(\GG)$) we have
\[
\nu_n((\upomega_{\xi}\tens\id)(\wW))\xrightarrow[n\to\infty]{}\upomega_{\xi}(P).
\]
Since the set $\bigl\{(\upomega_{\xi}\tens\id)(\wW)\st\xi\in\Ltwo(\GG)\bigr\}$ is linearly dense in $\C_0^\uu(\GG)$ and the sequence $(\nu_n)_{n\in\NN}$ is uniformly bounded we conclude that, for all $a\in\C_0^\uu(\GG)$ the sequence $\bigl(\nu_n(a)\bigr)_{n\in\NN}$ is convergent, i.e.~$(\nu_n)_{n\in\NN}$ is weak$^*$ convergent. Denoting its weak$^*$ limit by $\nu$  we have $P=(\id\otimes\nu)(\wW)$. Since $P^2=P=PP_\omega=P_{\omega}P=PP_\mu=P_{\mu}P$ we conclude that $\nu\staru\nu=\nu=\nu\staru\omega=\omega\staru\nu=\nu\staru\mu=\mu\staru\nu$.

As $\nu$ is a weak$^*$ limit of states, it satisfies $\|\nu\|\leq 1$. On the other hand, if $\nu\neq{0}$ then $\nu$, being an idempotent in the Banach algebra $\C_0^\uu(\GG)^*$, satisfies $\|\nu\|\geq{1}$. It follows that either $\nu=0$ or $\nu$ is an idempotent state, i.e.~$\nu\in\Idem_0(\GG)$.

Suppose further that $\nu\neq{0}$. We shall show that in this case $\sN_\nu=\sN_\omega\cap\sN_\mu$. Clearly, $E_\omega\comp{E_\nu}=E_{\omega\staru\nu}=E_\nu=E_{\mu\staru\nu}=E_{\mu}\comp{E_\nu}$, so that
\begin{equation}\label{cont1}
\sN_\nu\subset\sN_\omega\cap\sN_\mu.
\end{equation}
Since  the action of $\GG$ on $\sN_\nu$ is integrable, the same is true for the action on $\sN_\omega\cap\sN_\mu$. Let $x\in\Dom(\eta)\cap\sN_\omega\cap\sN_\mu$. Since  $\Ltwo(\sN_\omega\cap\sN_\mu)\subset\Ltwo(\sN_\omega)\cap\Ltwo(\sN_\mu)$ and $P_\nu$ is a projection onto $\Ltwo(\sN_\omega)\cap\Ltwo(\sN_\mu)$, we have $P_\nu\eta(x)=\eta(x)$, so by \eqref{eq1} with $\omega$ replaced by $\nu$ we get $E_\nu(x)=x$. Integrability of $\sN_\omega\cap\sN_\mu $ means that $\Dom(\eta)\cap\sN_\omega\cap\sN_\mu$ is dense in $\sN_\omega\cap\sN_\mu$ and thus $E_\nu(x)=x$ for all $x\in\sN_\omega\cap\sN_\mu$, i.e.~$\sN_\omega\cap\sN_\mu\subset\sN_\nu$ which together with \eqref{cont1} shows that $\sN_\nu=\sN_\omega\cap\sN_\mu$. In particular $\sN_\omega\cap\sN_\mu$ is non-zero and integrable. Finally, since $\Ltwo(\sN_\omega\cap\sN_\mu)\subset\Ltwo(\sN_\omega)\cap\Ltwo(\sN_\mu)$, this immediately implies $\Ltwo(\sN_\omega)\cap\Ltwo(\sN_\mu)\neq\{0\}$, so the three conditions in \eqref{vee3} are equivalent.
\end{proof}

We will denote the functional $\nu\in\Idem_0(\GG)$ described in Theorem \ref{vee} by $\omega\vee\mu$ and thereby define a binary operation on $\Idem(\GG)$. Note that $\omega\vee\mu=0$ if and only if $\Ltwo(\sN_\omega)\cap\Ltwo(\sN_\mu)=\{0\}$. Let us further extend the operation $\vee$ to $\Idem_0(\GG)$ by declaring $\omega\vee\mu=0$ if either (or both) functionals are zero. This way we obtain
\[
\Idem_0(\GG)\times\Idem_0(\GG)\ni(\omega,\mu)\longmapsto\omega\vee\mu\in\Idem_0(\GG)
\]
which is quite obviously associative and commutative. The quasi-subgroup corresponding to $\omega\vee\mu$ is the quasi-subgroup \emph{generated} by the quasi-subgroups related to $\omega$ and $\mu$ with the provision that $\omega\vee\mu=0$ when the corresponding quasi-subgroup is non-compact.

In particular, if $\GG$ is compact, then $\omega\vee\mu$ is a state (i.e.~it is non-zero) for all $\omega,\mu\in\Idem(\GG)$ because in this case $\C_0^\uu(\GG)=\C^\uu(\GG)$ is unital and hence the set of states is closed in the weak$^*$ topology.

\begin{remark}\label{order}
For $\alpha,\beta\in\Idem(\GG)$ the left coideal $\sN_{\alpha\wedge\beta}$ is the smallest left coideal containing both $\sN_\alpha$ and $\sN_\beta$ (it is automatically integrable and $\tau$-invariant). In particular $\alpha\wedge\beta\cleq\alpha$ and $\alpha\wedge\beta\cleq\beta$. Moreover $\alpha\wedge\beta$ is the largest (in the sense of the partial order $\cleq$) idempotent state dominated by both $\alpha$ and $\beta$:
\[
\alpha\wedge\beta=\sup\bigl\{\omega\in\Idem(\GG)\st\omega\cleq\alpha,\:\omega\cleq\beta\bigr\}.
\]

Similarly $\sN_{\alpha\vee\beta}$ is the largest integrable $\tau$-invariant left coideal contained in both $\sN_\alpha$ and $\sN_\beta$ (it is $\{0\}$ if their intersection is not integrable). Consequently $\alpha\cleq\alpha\vee\beta$, $\beta\cleq\alpha\vee\beta$ and $\alpha\vee\beta$ is the smallest (again in the sense of $\cleq$) element of $\Idem_0(\GG)$ dominating both $\alpha$ and $\beta$:
\[
\alpha\vee\beta=\inf\bigl\{\omega\in\Idem_0(\GG)\st\alpha\cleq\omega,\:\beta\cleq\omega\bigr\}.
\]
\end{remark}

\subsection{Modular law}\label{modLaw}

In the statement of the next theorem the symbol ``$\sigma-\text{\rm{c.l.s.}}$'' stands for ``$\sigma$-weakly closed linear span.''

\begin{theorem}\label{modlaw}
Let $\omega,\mu,\rho\in\Idem(\GG)$ be such that
\begin{enumerate}[\indent\rm(1)]
\item\label{modlaw1} $\rho\cleq\omega$,
\item $\rho\vee\mu=\rho\staru\mu$,
\item\label{modlaw3} $\sN_{\omega\wedge\mu}=(\sN_\omega\sN_\mu)^{\sigma-\text{\rm{c.l.s.}}}$.
\end{enumerate}
Then $\omega\wedge(\mu\vee\rho)=(\omega\wedge \mu)\vee\rho$.
\end{theorem}

Theorem \ref{modlaw} is modeled on \cite[Theorem 4.18]{alex}.

Before moving on to the proof of Theorem \ref{modlaw} let us note that the following conditions are in fact equivalent:
\begin{enumerate}[\indent(i)]
\item\label{muro1} $\rho\vee\mu=\rho\staru\mu$,
\item\label{muro2} $\mu\staru\rho\staru\mu=\rho\staru\mu$,
\item\label{muro3} $\mu\staru\rho=\rho\staru\mu$.
\end{enumerate}
Indeed, \eqref{muro1}$\Rightarrow$\eqref{muro2} follows, because $\rho\vee\mu$ dominates $\mu$, i.e.~$\mu\staru(\rho\vee\mu)=\rho\vee\mu$. Assume \eqref{muro2}. Then $P_\mu{P_\rho}P_\mu=P_\rho{P_\mu}$ and this implies that $P_\rho{P_\mu}=P_\mu{P_\rho}$, so we obtain \eqref{muro3}. Finally from \eqref{muro3} we get that $\rho\staru\mu$ is idempotent:
\[
(\rho\staru\mu)\staru(\rho\staru\mu)=\rho\staru\mu\staru\mu\staru\rho=\rho\staru\mu\staru\rho=\rho\staru\rho\staru\mu=\rho\staru\mu
\]
and it clearly dominates $\rho$ and $\mu$:
\[
\rho\staru(\rho\staru\mu)=\rho\staru\mu\quad\text{and}\quad\mu\staru(\rho\staru\mu)=\mu\staru\mu\staru\rho\staru=\mu\staru\rho=\rho\staru\mu.
\]
Moreover if $\lambda$ dominates $\rho$ and $\mu$ (i.e.~$\rho\staru\lambda=\lambda=\mu\staru\lambda$) then
\[
(\rho\staru\mu)\staru\lambda=\rho\staru(\mu\staru\lambda)=\rho\staru\lambda=\lambda,
\]
so $\rho\staru\mu\cleq\lambda$. It follows from Remark \ref{order} that $\rho\staru\mu=\rho\vee\mu$.

\begin{proof}[Proof of Theorem \ref{modlaw}]
The obvious relations $\omega\wedge\mu\cleq\omega$ and $\omega\wedge\mu\cleq\mu\cleq\rho\vee\mu$ imply
\[
\omega\wedge\mu\cleq\omega\wedge(\rho\vee\mu)
\]
(see Remark \ref{order}). Similarly $\rho\cleq\omega$ and $\rho\cleq\rho\wedge\mu$ imply
\[
\rho\cleq\omega\wedge(\rho\vee\mu).
\]
Therefore, once more by Remark \ref{order}, we obtain
\begin{equation}\label{jednastrona}
(\omega\wedge\mu)\vee\rho\cleq\omega\wedge(\rho\vee\mu).
\end{equation}

On the other hand, we know from the discussion preceding the proof that $E_{\mu\staru\rho\staru\mu}=E_{\rho\staru\mu}$ and hence $E_\mu\comp{E_\rho}\comp{E_\mu}=E_\rho\comp{E_\mu}$. It follows that $E_\rho(\sN_\mu)\subset\sN_\mu$, so
\begin{equation}\label{prodcont}
E_\rho(\sN_\mu)\subset\sN_\rho\cap\sN_\mu.
\end{equation}

Using this we will show the converse of \eqref{jednastrona}. Take $x\in(\sN_\omega\vee\sN_\mu)\cap\sN_\rho$. Then by assumption \eqref{modlaw3} $x$ is a $\sigma$-weak limit of the form
\[
x=\lim_{i\in{I}}\sum_{j=1}^{N_i}x_i^jy_i^j,
\]
with $x_i^j\in\sN_\omega$ and $y_i^j\in\sN_\mu$. Then, remembering that $E_\rho$ is a normal map, we compute
\[
x=E_\rho(x)=E_\rho\biggl(\lim_{i\in{I}}\sum_{j=1}^{N_i}x_i^jy_i^j\biggr)=\lim_{i\in{I}}\sum_{j=1}^{N_i}E_\rho(x_i^jy_i^j)=\lim_{i\in{I}}\sum_{j=1}^{N_i}x_i^jE_\rho(y_i^j)
\]
(note that $\sN_\omega\subset\sN_\rho$ by assumption \eqref{modlaw1} and hence the last passage above). Now all elements $x_i^jE_\rho(y_i^j)$ belong to $\sN_\omega\vee(\sN_\mu\cap\sN_\rho)$ by \eqref{prodcont}. It follows that
$x\in\sN_\omega\vee(\sN_\mu\cap\sN_\rho)$, which yields
\[
\omega\wedge(\rho\vee\mu)\cleq(\omega\wedge\mu)\vee\rho.
\]
\end{proof}

\begin{remark}
Condition \eqref{modlaw3} of Theorem \ref{modlaw} is satisfied e.g.~when $\omega$ or $\mu$ is of Haar type, see \cite[Remark 2.33]{alex}.
\end{remark}

\section{Open compact quasi-subgroups}\label{openSect}

Using the reducing morphism $\Lambda:\xymatrix@1{\C_0^\uu(\GG)\ar@{->>}[r]&\C_0(\GG)}\subset\Linf(\GG)$ we embed the space $\Linf(\GG)_*$ into $\C_0^\uu(\GG)^*$. It turns out that $\Linf(\GG)_*$ is a closed ideal of the Banach algebra $(\C_0^\uu(\GG)^*,\staru)$ (\cite[Proposition 5.3]{univ}). In this section we will consider idempotent states on $\GG$ which lie in this ideal. Denote
\[
\Idemn(\GG)=\Idem(\GG)\cap\Linf(\GG)_*.
\]
It is natural to refer to the compact quasi-subgroups corresponding to elements of $\Idemn(\GG)$ as \emph{open} compact quantum quasi-subgroups of $\GG$. In this context the next proposition generalizes the simple fact that if a subgroup of a topological group contains an open subgroup then it is itself open.

\begin{proposition}\label{opencleq}
For any $\omega,\mu\in\Idem(\GG)$  we have
\[
\left(
\begin{array}{c}
\omega\cleq\mu\\
\omega\in\Idemn(\GG)
\end{array}
\right)\;\Longrightarrow\;\Bigl(\mu\in\Idemn(\GG)\Bigr).
\]
\end{proposition}

\begin{proof}
Recall that $\Linf(\GG)_*$ is an ideal in $\C_0^\uu(\GG)^*$. It follows that $\mu=\omega\staru\mu\in\Linf(\GG)_*$, so $\mu\in\Idemn(\GG)$.
\end{proof}

\begin{corollary}
If $\GG$ is a discrete quantum group then $\Idemn(\GG)=\Idem(\GG)$.
\end{corollary}

\begin{proof}
The counit $\varepsilon$ of $\GG$ is normal. Consequently any $\omega\in\Idem(\GG)$ satisfies $\varepsilon\cleq\omega$, so $\omega\in\Idemn(\GG)$ by Proposition \ref{opencleq}.
\end{proof}

In other words any compact quasi-subgroup of a discrete quantum group is open (cf.~\cite{DKSS,open}).

For each $\omega\in\Idemn(\GG)$ consider its left kernel
\[
\sJ_\omega=\bigl\{x\in\Linf(\GG)\st\omega(x^*x)=0\bigr\}.
\]
It is well known that $\sJ_\omega$ is a $\sigma$-weakly closed left ideal in $\Linf(\GG)$ and hence it is of the form $\Linf(\GG)Q_\omega$ for a unique projection $Q_\omega\in\Linf(\GG)$. Clearly we have $\omega(x^*x)=0$ if and only if $xQ_\omega=x$. In particular $\omega(Q_\omega)=0$ and by Schwarz inequality $\omega(yQ_\omega)=0=\omega(Q_\omega{y})$ for all $y\in\Linf(\GG)$. Moreover, putting $Q_\omega^\perp=\I-Q_\omega$ we get
\begin{equation}\label{yQomega}
\omega(y)=\omega(yQ_\omega^\perp)=\omega(Q_\omega^\perp{y})=\omega(Q_\omega^\perp{y}Q_\omega^\perp),\qqquad{y}\in\Linf(\GG).
\end{equation}

\begin{remark}\label{remJ}
Let us note that if $x\in\Linf(\GG)$ is positive and $\omega(x)=0$ then $Q_\omega^\perp{x}=xQ_\omega^\perp=0$. Indeed, $x=y^*y$ for some $y\in\sJ_\omega=\Linf(\GG)Q_\omega$, so that $y=zQ_\omega$ for some $z$. Consequently $x=Q_\omega{z^*z}Q_\omega$.
\end{remark}

\begin{lemma}\label{L}
Let $\omega\in\Idemn(\GG)$. Then
\begin{enumerate}[\indent\rm(1)]
\item\label{L1} $x\in\sN_\omega$ if and only if $\Delta(x)(\I\tens{Q_\omega^\perp})=x\tens{Q_\omega^\perp}$,
\item\label{L2} $\Delta(Q_\omega)(Q_\omega^\perp\tens{Q_\omega^\perp})=0$,
\item\label{L3} $R(Q_\omega)=Q_\omega$ and $\tau_t(Q_\omega)=Q_\omega$ for all $t\in\RR$.
\end{enumerate}
\end{lemma}

\begin{proof}
Ad \eqref{L1}. Assume first that $\Delta(x)(\I\tens{Q_\omega^\perp})=x\tens{Q_\omega^\perp}$. Slicing this equality with $\omega$ over the right leg and using \eqref{yQomega} we find that
\[
\omega\staru{x}=\omega(Q_\omega^\perp)x=x
\]
i.e.~$x\in\sN_\omega$.

Conversely, suppose that $x\in\sN_\omega$. Then $x^*x\in\sN_\omega$ and hence $(\id\tens\omega)\bigl(\Delta(x^*x)\bigr)=x^*x$. In particular
\[
\begin{split}
(\id\tens\omega)&\bigl((\Delta(x)-x\tens\I)^*(\Delta(x)-x\tens\I)\bigr)\\
&=(\id\tens\omega)\bigl(\Delta(x^*x)-\Delta(x^*)(x\tens\I)-(x^*\tens\I)\Delta(x)+x^*x\tens\I\bigr)\\
&=x^*x-(\omega\staru{x})^*x-x^*(\omega\staru{x})+x^*x=0.
\end{split}
\]
It follows that for all positive $\mu\in\Linf(\GG)_*$ we have
\[
\omega\Bigl((\mu\tens\id)\bigl((\Delta(x)-x\otimes\I)^*(\Delta(x)-x\otimes\I)\bigr)\Bigr)=0
\]
Denote the positive element $(\mu\tens\id)\bigl((\Delta(x)-x\otimes\I)^*(\Delta(x)-x\otimes\I)\bigr)$ by $x_\mu$. We have $\omega(x_\mu)=0$, so $x_\mu{Q_\omega^\perp}=0$ by Remark \ref{remJ} and consequently
\[
\begin{split}
0&={x_\mu}Q_\omega^\perp=(\mu\tens\id)\bigl((\Delta(x)-x\otimes\I)^*(\Delta(x)-x\otimes\I)\bigr)Q_\omega^\perp\\
&=(\mu\tens\id)\Bigl(\bigl((\Delta(x)-x\otimes\I)^*(\Delta(x)-x\otimes\I)\bigr)(\I\tens{Q_\omega^\perp})\Bigr).
\end{split}
\]
Since this is true for all positive $\mu\in\Linf(\GG)_*$, we get $\bigl(\Delta(x)-x\otimes\I\bigr)^*\bigl(\Delta(x)-x\otimes\I\bigr)(\I\tens{Q_\omega^\perp})=0$, and thus
\[
(\I\tens{Q_\omega^\perp})\bigl(\Delta(x)-x\otimes\I\bigr)^*\bigl(\Delta(x)-x\otimes\I\bigr)(\I\tens{Q_\omega^\perp})=0.
\]
Therefore $\bigl(\Delta(x)-x\otimes\I\bigr)(\I\tens{Q_\omega^\perp})=0$, i.e.~$\Delta(x)(\I\tens{Q_\omega^\perp})=x\tens{Q_\omega^\perp}$.

Ad \eqref{L2}. Since
\[
(\omega\tens\omega)\Delta(Q_\omega)=(\omega\staru\omega)(Q_\omega)=\omega(Q_\omega)=0,
\]
we find that $(\id\tens\omega)\Delta(Q_\omega)$ is a positive element of $\ker{\omega}$, so by Remark \ref{remJ} we have
\[
Q_\omega^\perp\bigl((\id\tens\omega)\Delta(Q_\omega)\bigr)Q_\omega^\perp=0.
\]

It follows that for any positive $\mu\in\Linf(\GG)_*$
\[
0=\mu\Bigl(Q_\omega^\perp\bigl((\id\tens\omega)\Delta(Q_\omega)\bigr)Q_\omega^\perp\Bigr)
=\omega\Bigl((\mu\tens\id)\bigl((Q_\omega^\perp\tens\I)\Delta(Q_\omega)(Q_\omega^\perp\tens\I)\bigr)\Bigr),
\]
so again by the initial remark
\[
Q_\omega^\perp
\Bigl((\mu\tens\id)\bigl((Q_\omega^\perp\tens\I)\Delta(Q_\omega)(Q_\omega^\perp\tens\I)\bigr)\Bigr)
Q_\omega^\perp=0
\]
or in other words
\[
(\mu\tens\id)\bigl((Q_\omega^\perp\tens{Q_\omega^\perp})\Delta(Q_\omega)(Q_\omega^\perp\tens{Q_\omega^\perp})\bigr)=0.
\]
Since this is true for all positive $\mu\in\Linf(\GG)_*$, we obtain
\[
(Q_\omega^\perp\tens{Q_\omega^\perp})\Delta(Q_\omega)(Q_\omega^\perp\tens{Q_\omega^\perp})=0,
\]
and it follows that $\Delta(Q_\omega)(Q_\omega^\perp\tens{Q_\omega^\perp})=0$ because $\Delta(Q_\omega)$ is a projection.

Ad \eqref{L3}. By \cite[Proposition 4]{SaSk} we have $\omega\comp{R}=\omega=\omega\comp\tau_t$ for all $t$. In particular  $x\in\sJ_\omega$ if and only if $\tau_t(x)\in\sJ_\omega$ for all $t$ and hence $x\in\sJ_\omega$ if and only if $x\tau_t(Q_\omega)=x$ for all $t$. In particular for any $t$ we obtain $\tau_t(Q_\omega)=Q_\omega$. Similarly $x\in\sJ_\omega$ if and only $R(x^*)\in\sJ_\omega$ if and only if $R(x^*)Q_\omega = R(x^*)$ if and only if $R(Q_\omega)x^*=x^*$ if and only $xR(Q_\omega)=x$. Thus $R(Q_\omega)=Q_\omega$.
\end{proof}

Recall that a projection $P\in\Linf(\GG)$ is called \emph{group-like} if $\Delta(P)(\I\tens{P})=P\tens{P}$.

\begin{theorem}\label{thglp}
Let $P\in\Linf(\GG)$ be a projection satisfying $\Delta(P)(P^\perp\tens{P^\perp})=0$, $P\in\Dom(S)$ and $S(P)=P$. Then $P^\perp$ is group-like.
\end{theorem}

\begin{proof}
It suffices to show that $\Delta(P^\perp)(P\tens{P^\perp})=0$. Indeed the latter implies that
\[
\begin{split}
\Delta(P^\perp)(\I\tens{P^\perp})&=\Delta(P^\perp)\bigl((P+P^\perp)\tens{P^\perp}\bigr)\\
&=\Delta(P^\perp)(P\tens{P^\perp})+\Delta(P^\perp)(P^\perp\tens{P^\perp})\\
&=\Delta(P^\perp)(P\tens{P^\perp})+\Delta(\I-P)(P^\perp\tens{P^\perp})\\
&=\Delta(P^\perp)(P\tens{P^\perp})+(P^\perp\tens{P^\perp})-\Delta(P)(P^\perp\tens{P^\perp})=P^\perp\tens{P^\perp}
\end{split}
\]
because $\Delta(P)(P^\perp\tens{P^\perp})=0$ by assumption.

Obviously, as $\Delta(P)(P^\perp\tens{P^\perp})=0$, also for any $a,b\in\Dom(\eta)$
\[
(\psi\tens\id)\bigl(\Delta(a^*P)(P^\perp{b}\tens{P^\perp}\bigr)=0.
\]
or in other words
\[
(\psi\tens\id)\bigl(\Delta(a^*P)(P^\perp{b}\tens\I)\bigr)P^\perp=0.
\]
By the strong right invariance of $\psi$ (\cite[Definition 1.5(3)]{mnw}, cf.~also \cite[Proposition 5.24]{KV} and \cite[Definition 1.1(iv) b)]{mn}, but note different convention regarding the antipode) and Lemma \ref{L}\eqref{L3} the left hand side of the above equality is
\[
\begin{split}
S\Bigl((\psi\tens\id)\bigl((a^*P\tens\I)\Delta(P^\perp{b})\bigr)\Bigr)P^\perp
&=S\Bigl((\psi\tens\id)\bigl((a^*P\tens\I)\Delta(P^\perp{b})\bigr)\Bigr)S(P^\perp)\\
&=S\Bigl(P^\perp(\psi\tens\id)\bigl((a^*P\tens\I)\Delta(P^\perp{b})\bigr)\Bigr)\\
&=S\Bigl((\psi\tens\id)\bigl((a^*P\tens{P^\perp})\Delta(P^\perp{b})\bigr)\Bigr).
\end{split}
\]
Hence we obtain
\[
(\psi\tens\id)\bigl((a^*P\tens{P^\perp})\Delta(P^\perp{b})\bigr)=0,\qqquad{a,b}\in\Dom(\eta).
\]
i.e.
\[
(\psi\tens\id)\bigl((a^*\tens\I)(P\tens{P^\perp})\Delta(P^\perp)\Delta(b)\bigr)=0,\qqquad{a,b}\in\Dom(\eta).
\]
which implies $(P\tens{P^\perp})\Delta(P^\perp)=0$, so also $\Delta(P^\perp)(P\tens{P^\perp})=0$.
\end{proof}

\begin{corollary}
Let $\omega\in\Idemn(\GG)$. Then $Q_\omega^\perp$ is a group-like projection.
\end{corollary}

\begin{proof}
Statements \eqref{L3} and \eqref{L2} of Lemma \ref{L} show that the projection $P=Q_\omega^\perp$ satisfies the assumptions of Theorem \ref{thglp} and consequently $Q_\omega^\perp$ is group-like.
\end{proof}

\begin{remark}\label{sigmaRem}
We have established that $Q_\omega^\perp$ is a group-like projection invariant under $R$ and $\tau$. Therefore, by \cite[Theorem 4.3 and Theorem 3.1]{FK}, $Q_\omega^\perp$ is invariant under the modular group $\sigma^\psi$ of $\psi$:
\[
\sigma^\psi_t(Q_\omega^\perp)=Q_\omega^\perp,\qqquad{t}\in\RR.
\]
\end{remark}

\begin{theorem}\label{Qmin}
For any $\omega\in\Idemn(\GG)$ the projection $Q_\omega^\perp$ belongs to $\sN_\omega$. Moreover $Q_\omega^\perp$ is minimal and central in $\sN_\omega$ and $Q_\omega^\perp\in\C_0(\GG)\cap\Dom(\eta)$.
\end{theorem}

\begin{remark}
Note that if $\GG$ is a locally compact quantum group and $Q$ a non-zero positive element in $\Dom(\eta)$ invariant under $\tau$ then the scaling constant $\lambda$ of $\GG$ must be equal to one. In particular it follows from Lemma \ref{L}\eqref{L3} and the fact that $Q_\omega^\perp$ is (square-) integrable that if $\Idemn(\GG)\neq\emptyset$ then $\lambda=1$. In particular the quantum ``$az+b$'' groups (\cite{azb,nazb}) do not admit normal idempotent states.
\end{remark}

\begin{remark}\label{finsupp}
Another consequence of integrability of $Q_\omega^\perp$ is that any idempotent state on a discrete quantum group $\GG$ has ``finite support'' in the sense that it is zero on almost all direct summands of the algebra
\[
\M\bigl(\operatorname{c}_0(\GG)\bigr)=\bigoplus_{\alpha}M_{n_\alpha}
\]
(this is a $\ell^\infty$-direct sum, cf.~\cite[Section 2]{qlor}). Indeed, the projection $Q_\omega^\perp$ must have non-zero components in only finitely many summands due to the formula for $\psi$ (\cite[Equation (2.13)]{qlor}) which gives $\psi(Q_\omega^\perp)$ as a sum of traces of components of $Q_\omega^\perp$ multiplied by appropriate quantum dimensions which are all greater than or equal to $1$.
\end{remark}

\begin{proof}[Proof of Theorem \ref{Qmin}]
Once we know that $Q_\omega^\perp$ is a group-like projection, we see immediately by Lemma \ref{L}\eqref{L1} that $Q_\omega^\perp\in\sN_\omega$.

Now take arbitrary $x\in\sN_\omega$. Then, again by Lemma \ref{L}\eqref{L1}, we have
\[
\Delta(x)(\I\tens{Q_\omega^\perp})=x\tens{Q_\omega^\perp}\quad\text{and}\quad(\I\tens{Q_\omega^\perp})\Delta(x)=x\tens{Q_\omega^\perp}.
\]
(with the latter equality obtained from the former applied to $x^*$). It follows that for any $\mu\in\Linf(\GG)_*$ we have
\[
(\mu\tens\id)\bigl(\Delta(x)\bigr)Q_\omega^\perp=\mu(x)Q_\omega^\perp=Q_\omega^\perp(\mu\tens\id)\bigl(\Delta(x)\bigr).
\]
Recall now that by the weak form of Podle\'s condition (\cite[Corollary 2.7]{proj}) elements of the form $(\mu\tens\id)\bigl(\Delta(x)\bigr)$ span a $\sigma$-weakly dense subset of $\sN_\omega$, and consequently $Q_\omega^\perp$ is a minimal and central projection in $\sN_\omega$.

Now $\sN_\omega$ is also integrable, so the set of elements of $\sN_\omega$ which are square-integrable with respect to the right Haar measure is $\sigma$-weakly dense in $\sN_\omega$. In particular there exists square-integrable element $x\in\sN_\omega$ such that $Q_\omega^\perp{x}\neq{0}$. The set $\Dom(\eta)\cap\sN_\omega$ is also a left ideal $\in\sN$ so $Q_\omega^\perp{x}\in\Dom(\eta)$. This element is at the same time proportional to $Q_\omega^\perp$ (by its minimality) and non-zero. It follows that $Q_\omega^\perp\in\Dom(\eta)$.

Since $Q_\omega^\perp$ is a group-like projection preserved by the scaling group, by the results of \cite{FK} mentioned in Section \ref{IdempSt}, there exists an idempotent state $\mu$ on $\hh{\GG}$ such that $Q_\omega^\perp= (\mu\tens\id)(\Ww)$. Thus $Q_\omega^\perp\in\M(\C_0(\GG))$. Again, since $Q_\omega^\perp$ is a minimal projection in $\sN_\omega$ and $E_\omega\bigl(\C_0(\GG)\bigr)\subset\sN_\omega\cap\C_0(\GG)$ is a $\sigma$-weakly dense subset of $\sN_\omega$, there must exist $a\in\C_0(\GG)\cap\sN_\omega$ such that $Q_\omega^\perp=aQ_\omega^\perp$ and the latter element belongs to $\C_0(\GG)\M\bigl(\C_0(\GG)\bigr)=\C_0(\GG)$.
\end{proof}

\begin{theorem}\label{wzor}
For any $\omega\in\Idemn(\GG)$ we have
\[
\omega(x)=\frac{\psi(Q_\omega^\perp\;\!x\;\!Q_\omega^\perp)}{\psi(Q_\omega^\perp)},\qqquad{x}\in\Linf(\GG).
\]
\end{theorem}

\begin{proof}
Define a functional $\theta$ on $\Linf(\GG)$ by
\[
\theta(x)=\frac{\psi(Q_\omega^\perp\;\!x\;\!Q_\omega^\perp)}{\psi(Q_\omega^\perp)},\qqquad{x}\in\Linf(\GG).
\]
Clearly $\theta$ is a normal state. Moreover $\theta$ is idempotent because
{\allowdisplaybreaks
\begin{align*}
(\theta\staru\theta)(x)&=(\theta\tens\theta)\bigl(\Delta(x)\bigr)\\
&=\tfrac{1}{\psi(Q_\omega^\perp)^2}(\psi\tens\psi)\bigl((Q_\omega^\perp\tens{Q_\omega^\perp})\Delta(x)(Q_\omega^\perp\tens{Q_\omega^\perp})\bigr)\\
&=\tfrac{1}{\psi(Q_\omega^\perp)^2}(\psi\tens\psi)\bigl((\I\tens{Q_\omega^\perp})\Delta(Q_\omega^\perp)\Delta(x)\Delta(Q_\omega^\perp)(Q_\omega^\perp\tens\I)\bigr)\\
&=\tfrac{1}{\psi(Q_\omega^\perp)^2}(\psi\tens\psi)\bigl((\I\tens{Q_\omega^\perp})\Delta(Q_\omega^\perp\;\!x\;\!Q_\omega^\perp)(Q_\omega^\perp\tens\I)\bigr)\\
&=\tfrac{1}{\psi(Q_\omega^\perp)^2}\psi\Bigl((\psi\tens\id)\bigl((\I\tens{Q_\omega^\perp})\Delta(Q_\omega^\perp\;\!x\;\!Q_\omega^\perp)(Q_\omega^\perp\tens\I)\bigr)\Bigr)\\
&=\tfrac{1}{\psi(Q_\omega^\perp)^2}\psi\Bigl(Q_\omega^\perp(\psi\tens\id)\bigl(\Delta(Q_\omega^\perp\;\!x\;\!Q_\omega^\perp)\bigr)Q_\omega^\perp\Bigr)\\
&=\tfrac{1}{\psi(Q_\omega^\perp)^2}\psi\bigl(Q_\omega^\perp\psi(Q_\omega^\perp\;\!x\;\!Q_\omega^\perp)Q_\omega^\perp\bigr)\\
&=\tfrac{\psi(Q_\omega^\perp\;\!x\;\!Q_\omega^\perp)}{\psi(Q_\omega^\perp)^2}\psi(Q_\omega^\perp)=\theta(x)
\end{align*}
}(where in the third equality we used the fact that $Q_\omega^\perp$ is a group-like projection invariant under $R$ and the seventh equality uses right invariance of $\psi$). Moreover $Q_\theta=Q_\omega$. Now, as $Q_\theta$ determines $\sN_\theta$ (cf.~Lemma \ref{L}\eqref{L1}), we also have $\sN_\theta=\sN_\omega$, so $\theta=\omega$ by the results gathered in Section \ref{IdempSt}.
\end{proof}

The next theorem uses co-duality for coideals defined e.g.~in \cite[Section 3]{eqhs}: given a left coideal $\sN\subset\Linf(\GG)$ the relative commutant
\[
\dd{\sN}=\sN'\cap\Linf(\hh{\GG})
\]
is a left coideal in $\Linf(\hh{\GG})$ called the \emph{co-dual} of $\sN$. We have $\dd{\dd{\sN}}=\sN$.

We will also use the following notion: a functional $\theta$ on $\Linf(\GG)$ is called \emph{$\Ltwo$-bounded} if there exists $\xi\in\Ltwo(\GG)$ such that
\[
\theta(x)=\is{\xi}{\eta(x)},\qqquad{x}\in\Dom(\eta).
\]
By \cite[Proposition 6.5(3)]{mnw} if $\theta$ and $\theta^*\comp{S}$ are $\Ltwo$-bounded then $(\id\tens\theta)(\ww)$ is square-integrable with respect to $\hh{\psi}$. Consider now a normal idempotent state $\omega$ on $\Linf(\GG)$ and take $x\in\Dom(\eta)$. By Theorem \ref{wzor} and the fact that $Q_\omega^\perp$ is $\sigma^\psi$-invariant (Remark \ref{sigmaRem}) we have
\[
\omega(x)=\frac{\psi(Q_\omega^\perp\;\!x\;\!Q_\omega^\perp)}{\psi(Q_\omega^\perp)}=\frac{\psi(Q_\omega^\perp\:\!x)}{\psi(Q_\omega^\perp)}=\is{\xi}{\eta(x)},
\]
for $\xi=\tfrac{1}{\psi(Q_\omega^\perp)}\eta(Q_\omega^\perp)$. It follows that $\omega$ is $\Ltwo$-bounded. Moreover $\omega^*\comp{S}=\omega\comp{S}=\omega$ by Lemma \ref{L}\eqref{L3}. Thus
\[
P_\omega=(\id\tens\omega)(\ww)
\]
is square-integrable with respect to $\hh{\psi}$.

\begin{theorem}\label{thmcod}
Let $\omega\in\Idemn(\GG)$. Then there exists a unique $\dd{\omega}\in\Idemn(\hh{\GG})$ such that $\dd{\sN_\omega}=\sN_{\dd{\omega}}$. Moreover $(\dd{\omega}\tens\id)(\ww)=Q_\omega^\perp$ and $\dd{\dd{\omega}}=\omega$.
\end{theorem}

\begin{proof}
Remarks preceding the theorem show that $P_\omega$ is integrable with respect to $\hh{\psi}$. Define a normal state $\dd{\omega}$ on $\Linf(\hh{\GG})$ by setting
\[
\dd{\omega}(y)=\frac{\hh{\psi}(P_\omega\;\!y\;\!P_\omega)}{\hh{\psi}(P_\omega)},\qqquad{y}\in\Linf(\hh{\GG}).
\]
It is easy to see that $\dd{\omega}$ is idempotent ($P_\omega$ is a group-like projection, see the proof of Theorem \ref{wzor}). Additionally, by \cite[Theorem 3.1]{FK}
we have
\begin{equation}\label{ddNom}
\dd{\sN_{\omega}}=\bigl\{y\in\Linf(\hh{\GG}){\bigl.\st\bigr.}\hh{\Delta}(y)(\I\tens{P_\omega})=y\tens{P_\omega}\bigr\}.
\end{equation}
Also it is clear that
\begin{equation}\label{PomQ}
P_\omega=Q_{\dd{\omega}}^\perp,
\end{equation}
so comparing \eqref{ddNom} with Lemma \ref{L}\eqref{L1} (applied to $\dd{\omega}$) we immediately obtain
\[
\dd{\sN_{\omega}}=\sN_{\dd{\omega}}.
\]
In view of the one-to-one correspondence between integrable coideals and idempotent states described in Section \ref{IdempSt}, this determines $\dd{\omega}$ uniquely and guarantees that $\dd{\dd{\omega}}=\omega$. Finally substituting $\dd{\omega}$ for $\omega$ in \eqref{PomQ} we get $Q_\omega^\perp=P_{\dd{\omega}}=(\dd{\omega}\tens\id)(\ww)$.
\end{proof}

The next theorem is in some sense a converse to Theorem \ref{thmcod}.

\begin{theorem}\label{conv_thmcod}
Let $\omega\in\Idem(\GG)$ be such that $\dd{\sN_\omega}$ is integrable. Then $\omega\in\Idemn(\GG)$.
\end{theorem}

\begin{proof}
The coideal $\sN_\omega$ is $\tau$-invariant, so $\dd{\sN_\omega}$ is $\hh{\tau}$-invariant by
\cite[Proposition 3.2]{ext}. As it is integrable by assumption, $\dd{\sN_\omega}$ must be of the form $\sN_{\breve{\omega}}$ for some $\breve{\omega}\in\Idem(\hh{\GG})$. We will show that $\breve{\omega}$ is normal and hence $\omega=\dd{\breve{\omega}}$ will be shown to be normal. To that end we will extensively use \cite[Theorem 3.1]{FK}.

Consider the projection $P_\omega$ corresponding to $\omega$. Then $P_\omega\in\dd{\sN_\omega}=\sN_{\breve{\omega}}$ and it is minimal and central in $\sN_{\breve{\omega}}$. It follows that it must be integrable, so we can define a normal state $\theta$ on $\Linf(\hh{\GG})$ by
\[
\theta(y)=\frac{\hh{\psi}(P_\omega\;\!y\;\!P_\omega)}{\hh{\psi}(P_\omega)},\qqquad{y}\in\Linf(\hh{\GG}).\]
The state $\theta$ is idempotent, as $P_\omega$ is group-like. Now take any $x\in\sN_{\breve{\omega}}$, so that
\[
\Delta_{\hh{\GG}}(x)(\I\tens{P_\omega})=x\tens{P_\omega}.
\]
Then we have
\[
\begin{split}
E_\theta(x)&=\tfrac{1}{\hh{\psi}(P_\omega)}(\id\tens\hh{\psi})\bigl((\I\tens{P_\omega})\Delta_{\hh{\GG}}(x)(\I\tens{P_\omega})\bigr)\\
&=\tfrac{1}{\hh{\psi}(P_\omega)}(\id\tens\hh{\psi})\bigl(x\tens{P_\omega}\bigr)=x.
\end{split}
\]
This shows that $E_\theta\comp{E_{\breve{\omega}}}=E_{\breve{\omega}}$, so $\theta\cleq\breve{\omega}$ by Lemma \ref{lemcomp}. Since $\theta$ is normal, so is $\breve{\omega}$ by Proposition \ref{opencleq}.
\end{proof}

\begin{remark}
The above result establishes a bijective correspondence between open compact quasi-subgroups of $\GG$ and those of $\hh{\GG}$. It is related to a special case of \cite[Theorem 7.2]{open} which gives a bijective correspondence between normal open quantum subgroups of $\GG$ and normal compact quantum subgroups of $\hh{\GG}$.
\end{remark}

\begin{corollary}
Let $\omega,\nu\in\Idemn(\GG)$. Then $\omega\cleq\mu$ if and only if $Q_\omega^\perp\leq{Q_\mu^\perp}$.
\end{corollary}

\begin{proof}
By Lemma \ref{lemcomp} we have $\omega\cleq\mu$ if and only if $\sN_\mu\subset\sN_\omega$ which is, in turn, equivalent to $\dd{\sN_\omega}\subset\dd{\sN_\mu}$ and thus to $P_{\dd{\omega}}\leq{P_{\dd{\mu}}}$ (again by Lemma \ref{lemcomp}). By Theorem \ref{thmcod} this means that $Q_\omega^\perp\leq{Q_\mu^\perp}$.
\end{proof}

Open quasi-subgroups of a \emph{compact} quantum group have a more precise characterization:

\begin{theorem}\label{findim}
Let $\GG$ be a compact quantum group and $\omega\in\Idem(\GG)$. Then $\omega\in\Idemn(\GG)$ if and only if $\dim{\sN_\omega}<+\infty$.
\end{theorem}

\begin{proof}
Assume that $\omega\in\Idemn(\GG)$. Then $\sN_\omega$ admits a minimal central projection (namely $Q_\omega^\perp$, cf.~Theorem \ref{Qmin}). It follows that $\sN_\omega$ has a finite dimensional direct summand. Furthermore the natural action of $\GG$ on $\sN_\omega$ is clearly ergodic, so by \cite[Theorem 3.4]{dCKSS} we have $\dim{\sN_\omega}<+\infty$.

Conversely, if $\dim(\sN_\omega)<+\infty$ then $\eta(\sN_\omega)\subset\Ltwo(\GG)$ is finite dimensional. Thus $P_\omega=(\id\tens\omega)(\wW)=(\id\tens\omega)(\ww)$, which is the projection onto $\Ltwo(\sN_\omega)$, intersects  only finitely many isotypical components of $\Linf(\GG)$ (cf.~\cite[Section 2]{qlor}). This shows that $\hh{\psi}(P_\omega)<\infty$ (cf.~the formula for right Haar measure on a discrete quantum group \cite[Equation (2.13)]{qlor}). We can therefore repeat the steps of the proof of Theorem \ref{thmcod} to find that
\[
\dd{\sN_\omega}=\sN_{\dd{\omega}}
\]
where $\dd{\omega}\in\Idemn(\hh{\GG})$. Now, by Theorem \ref{thmcod} applied to $\hh{\GG}$, we find that $\omega\in\Idemn(\GG)$.
\end{proof}

\begin{remark}
Let us note that the first part of the proof of Theorem \ref{findim} can also be performed without using \cite[Theorem 3.4]{dCKSS}. Indeed, if $\GG$ is a compact quantum group and $\omega\in\Idemn(\GG)$ then we already know from Theorem \ref{thmcod} and its proof that
\[
\dd{\omega}(y)=\frac{\hh{\psi}(P_\omega\;\!y\;\!P_\omega)}{\hh{\psi}(P_\omega)},\qqquad{y}\in\ell^\infty(\hh{\GG}).
\]
By Remark \ref{finsupp} the projection $P_\omega$ is finite dimensional, and since $P_\omega$ is also the projection onto $\Ltwo(\sN_\omega)$ we see that $\dim{\sN_\omega}<+\infty$.
\end{remark}

It is well known that a subgroup of a topological group is open if and only if the corresponding quotient space is discrete. The next corollary characterizes open quasi-subgroups of compact quantum groups in an analogous manner.

\begin{corollary}
Let $\GG$ be a compact quantum group and $\omega\in\Idem(\GG)$. Then $\sN_\omega$ has a finite dimensional direct summand if and only if $\omega\in\Idemn(\GG)$.
\end{corollary}

\begin{proof}
If $\omega\in\Idemn(\GG)$ then $\dim{\sN_\omega}<+\infty$, so $\sN_\omega$ is a direct sum of matrix algebras. Conversely, if $\sN_\omega$ has a finite dimensional direct summand then by \cite[Theorem 3.4]{dCKSS} we have $\dim{\sN_\omega}<+\infty$ so $\omega\in\Idemn(\GG)$ by Theorem \ref{findim}.
\end{proof}

As another result of similar nature consider the following proposition:

\begin{proposition}\label{findimq}
Let $\GG$ be a locally compact quantum group and let $\omega\in\Idem(\GG)$ be such that $\dim{\sN_\omega}<+\infty$. Then $\GG$ is compact and consequently $\omega\in\Idemn(\GG)$.
\end{proposition}

\begin{proof}
We know that the coideal $\sN_\omega$ is integrable, so if $\dim{\sN_\omega}<+\infty$, we have $\psi(\I)<+\infty$, so that $\GG$ is compact. The last statement follows from Theorem \ref{findim}.
\end{proof}

Proposition \ref{findimq} corresponds to the elementary fact that if a quotient by a compact subgroup is finite then the original group must also be compact.

\begin{theorem}
Let $\GG$ be a  compact quantum group and $\sN\subset\Linf(\GG)$ a finite dimensional coideal. Then there exists $\omega\in\Idemn(\GG)$ such that $\sN=\sN_\omega$. In particular $\sN$ is $\tau$-invariant.
\end{theorem}

\begin{proof}
Let $P\in\B(\Ltwo(\GG))$ be the projection onto $\Ltwo(\sN)$. We first show that $P$ belongs to $\ell^\infty(\hh{\GG})$. Indeed: for $\mu\in\C(\GG)^*$ the operator $x_\mu$ which defined by
\[
x_\mu\eta(a)=\eta\bigl((\mu\tens\id)\Delta(a)\bigr),\qqquad{a}\in\C(\GG)
\]
belongs to $\ell^\infty(\hh{\GG})'$ and the span of such operators is a $\sigma$-weakly dense subspace of $\ell^\infty(\hh{\GG})'$. Furthermore for any $a\in\C(\GG)$ and $\mu\in\C(\GG)^*$ we have
\[
Px_\mu\eta(a)=P\eta\bigl((\mu\tens\id)\Delta(a)\bigr)=\eta\bigl((\mu\tens\id)\Delta(a)\bigr)=x_\mu\eta(a),
\]
since $\eta\bigl((\mu\tens\id)\Delta(a)\bigr)$ clearly belongs to $\Ltwo(\sN)$. It follows that $Px_\mu{P}=x_\mu{P}$ and hence $P$ commutes with $\ell^\infty(\hh{\GG})'$.

Now we show that $P$ is a group-like projection in $\ell^\infty(\hh{\GG})$. Obviously it is enough to show that $\hh{\Delta}^{\text{\tiny{op}}}(P)(P\tens\I)=P\tens{P}$, i.e.
\begin{equation}\label{PP}
\ww^*(\I\tens{P})\ww(P\tens\I)=P\tens{P}.
\end{equation}
This follows from the following calculation: take $x\in\sN$ and $\xi\in\Ltwo(\GG)$. Then, since $P\in\sN'$, denoting $\eta(\I)$ by $\Omega_\psi$ we obtain
\[
\begin{split}
\ww^*(\I\tens{P})\ww(P\tens\I)\bigl(\eta(x)\tens\xi\bigr)&=\ww^*(\I\tens{P})\ww\bigl(\eta(x)\tens\xi\bigr)\\
&=\ww^*(\I\tens{P})\bigl(\Delta(x)(\Omega_\psi\tens\xi)\bigr)\\
&=\ww^*\bigl(\Delta(x)(\Omega_\psi\tens{P\xi})\bigr)=\eta(x)\tens{P\xi}
\end{split}
\]
and \eqref{PP} follows.

The next step is to note that $P$ is integrable for $\hh{\psi}$ (the argument is identical to that used in the proof of Theorem \ref{findim}). It follows that we can define a functional $\mu$ on $\ell^\infty(\hh{\GG})$ by
\[
\mu(y)=\frac{\hh{\psi}(P\;\!y\;\!P)}{\hh{\psi}(P)},\qqquad{y}\in\ell^\infty(\hh{\GG})
\]
and observe that it is clearly a normal state and it is idempotent (because $P$ is a group-like projection). Moreover $P=Q_\mu^\perp$. Thus, by Theorem \ref{thmcod} $P=P_{\dd{\mu}}$ and consequently, setting $\omega=\dd{\mu}$, we obtain $\sN=\sN_\omega$.
\end{proof}

Let us note that although intersection of open quantum subgroups of a locally compact quantum group is open (\cite[Corollary 2.29]{alex}), this no longer needs to be true for quasi-subgroups. More precisely let $\GG$ be a locally compact quantum group and let $\omega_1,\omega_2\in\Idemn(\GG)$ (i.e.~both states correspond to open compact quasi-subgroups). Then it need not be the case that $\omega_1\wedge\omega_2\in\Idemn(\GG)$. Indeed, consider a non-compact locally compact quantum group $\KK$ generated by two compact open subgroups $\HH_1$ and $\HH_2$ (cf.~\cite[Section 3]{Hopf}). Let $\mu_1$ and $\mu_2$ be the corresponding idempotent states of Haar type on $\C_0^\uu(\KK)$. Then first of all we have
\[
\sN_{\mu_i}=\Linf(\KK/\HH_i),\qqquad{i}=1,2.
\]
Put $\GG=\hh{\KK}$ and $\omega_i=\dd{\mu_i}$ for $i=1,2$. Then $\omega_1,\omega_2\in\Idemn(\GG)$ and
\[
\sN_{\omega_i}=\dd{\sN_{\mu_i}}=\dd{\Linf(\KK/\HH_i)}=\Linf(\hh{\HH_i}),\qqquad{i}=1,2.
\]
Furthermore, since $\KK=\hh{\GG}$ is generated by $\HH_1$ and $\HH_2$, we have
\[
\sN_{\omega_1\wedge\omega_2}=\sN_{\omega_1}\vee\sN_{\omega_2}=\Linf(\hh{\HH_1})\vee\Linf(\hh{\HH_2})=\Linf(\GG),
\]
so $\omega_1\wedge\omega_2$ must be the counit of $\GG$. However the latter is normal if and only if $\GG$ is discrete which is not the case since $\KK$ is not compact. It follows that $\omega_1\wedge\omega_2$ is not normal.

The next theorem sheds more light on this situation.

\begin{theorem}
Let $\omega_1,\omega_2\in\Idemn(\GG)$. Then $\omega_1\wedge\omega_2\in\Idemn(\GG)$ if and only if $\dd{\omega_1}\vee\dd{\omega_2}\neq{0}$.
\end{theorem}

\begin{proof}
Assume first that $\omega_1\wedge\omega_2$ is a normal state. Then
\[
\sN_{\dd{\omega_1\wedge\omega_2}}=\dd{\sN_{\omega_1\wedge\omega_2}}=\dd{\sN_{\omega_1}\vee\sN_{\omega_2}}=\dd{\sN_{\omega_1}}\cap\dd{\sN_{\omega_2}}=\sN_{\dd{\omega_1}\vee\dd{\omega_2}}
\]
so that $\dd{\omega_1}\vee\dd{\omega_2}=\dd{\omega_1\wedge\omega_2}\neq{0}$.

Conversely, if $\dd{\omega_1}\vee\dd{\omega_2}\neq{0}$ then $\dd{\omega_1}\vee\dd{\omega_2}\in\Idemn(\hh{\GG})$, so
\[
\omega_1\wedge\omega_2=\dd{\dd{\omega_1}\vee\dd{\omega_2}}
\]
belongs to $\Idemn(\GG)$, i.e.~it is normal.
\end{proof}

\subsection*{Acknowledgment}

The authors wish to thank Adam Skalski for helpful comments and the referee for suggesting improvements leading, in particular, to Theorem \ref{conv_thmcod}. Research presented in this paper was partially supported by the NCN (National Science Center, Poland) grant no.~2015/17/B/ST1/00085.

\end{document}